\UseRawInputEncoding
\documentclass[12pt,a4paper]{amsart}
\usepackage{amsmath, amsthm, amscd, amsfonts,graphicx}
\usepackage[matrix,arrow]{xy}
 \newtheorem{proposition}{Proposition}[section]
\newtheorem{theorem}{Theorem}[section]
\newtheorem{lemma}{Lemma}[section]

\newtheorem{example}{Example}[section]

\theoremstyle{definition}
\newtheorem{definition}{Definition}[section]

\newtheorem{remark}{Remark}[section]

\begin{document}

\title[A pretorsion theory for categories]{A pretorsion theory for the category of all categories}

 \author[Jo\~{a}o J. Xarez]{Jo\~{a}o J. Xarez}



 \address{CIDMA - Center for Research and Development in Mathematics and Applications,
Department of Mathematics, University of Aveiro, Portugal.}

 \email{xarez@ua.pt}




 \begin{abstract}
A pretorsion theory for the category of all categories is presented. The associated prekernels and precokernels are calculated for every functor.
\end{abstract}

\maketitle

\section{Introduction}
\label{sec:Introduction}

In \cite{stable category} it was shown that, in the category $Preord$ of preordered sets, there is a natural notion of pretorsion theory, in which the partially ordered sets are the torsion-free objects and the sets endowed with an equivalence relation are the torsion objects. The notion of pretorsion theory given in \cite{stable category} generalizes the notion of torsion theories for pointed categories given in \cite{JanTh:torsion theories}.\\

In this paper we give what can be seen as an extension of the pretorsion theory for preordered sets to categories. A torsion-free object is now a category whose image by the well known reflection $Cat\rightarrow Preord$ (cf. \cite{X:ml}) is a partially ordered set. A torsion object is in turn the category whose image by the same functor is an equivalence relation. They are called below respectively the antisymmetric and the symmetric categories (see the beginning of Section \ref{sec:symmetric, antisymmetric categories}).\\

In this way, the trivial objects which were sets for $Preord$ are now the collections of monoids for the category of all categories $Cat$.\\

To assert that in fact this is a pretorsion theory, it was necessary to characterize prekernels and precokernels (corresponding to kernels and cokernels in torsion theories). The crucial result in this paper being the construction of precokernels in $Cat$ (see Proposition \ref{prop:Catmon-Precokernels}).

\section{Pretorsion theory for a general category}\label{sec:pretorsion_theory}

$\textsc{Data}$\\

Consider two full replete\footnote{I.e., whose objects are closed under isomorphisms.} subcategories $\mathcal{T}$ and $\mathcal{F}$ of $\mathcal{C}$. They are called respectively the \textit{torsion} and the \textit{torsion-free subcategories}.

Set $\mathcal{Z}=\mathcal{T}\bigcap \mathcal{F}$, the full subcategory of $\mathcal{C}$ determined by the objects which are both in $\mathcal{T}$ and in $\mathcal{F}$. These objects are called \textit{trivial}.

Let $Triv_\mathcal{Z}(X,Y)$ be the set of all morphisms $X\rightarrow Y$ in $\mathcal{C}$ that factor through an object of $\mathcal{Z}$. Such morphisms will be called $\mathcal{Z}$\textit{-trivial} (or simply \textit{trivial}, if there is no doubt about the subcategories considered). Notice that these trivial morphisms form an \textit{ideal of morphisms} in the sense of \cite{Ehresmann}.\\

$\textsc{End of data}$\\

The following definitions, proposition and example are to be considered in the context of the data just presented.

\begin{definition} \label{def:prekernel and precokernel}

A morphism $k:X\rightarrow A$ is a $\mathcal{Z}$-\textit{prekernel} (or simply, \textit{prekernel}) of a morphism $f:A\rightarrow A'$ if the following two conditions are satisfied:
\begin{enumerate}
\item the composite $f\circ k$ is a trivial morphism;
\item whenever $\lambda :Y\rightarrow A$ is a morphism in $\mathcal{C}$ and $f\circ\lambda$ is trivial, then there exists a unique morphism $\lambda ':Y\rightarrow X$ in $\mathcal{C}$ such that $\lambda =k\circ\lambda '$.
\end{enumerate}
Dually, one obtains the notion of ($\mathcal{Z}$-)\textit{precokernel}.

\end{definition}

\begin{proposition}\label{prop:Galois_connection}

Suppose that $\mathcal{Z}$-prekernels and $\mathcal{Z}$-precokernels exist in $\mathcal{C}$.
Then, given any morphism $f$ in $\mathcal{C}$,
$$preker(precoker(preker f))\cong preker f$$ $$and$$ $$precoker(preker(precoker f))\cong precoker f,$$
\noindent where $preker f$ stands for the $\mathcal{Z}$-prekernel of $f$, and $precoker f$ stands for the $\mathcal{Z}$-precokernel of $f$.

\end{proposition}

\begin{proof}
This result is a consequence of the obvious Galois connections associated to each object in $\mathcal{C}$ (cf. \cite[\S VIII.1]{SM:cat}, for the classic and similar case of kernel and cokernel).

\end{proof}

\begin{definition} \label{def:short preexact sequence}

It is said that

$$
\begin{picture}(100,0)

\put(0,0){$A$}\put(50,0){$B$}\put(100,0){$C$,}

\put(13,5){\vector(1,0){35}}\put(63,5){\vector(1,0){35}}

\put(25,12){$f$}\put(75,12){$g$}

\end{picture}
$$ is a \textit{short} $\mathcal{Z}$-\textit{preexact sequence} (or simply, \textit{short preexact sequence}) in $\mathcal{C}$ if $f$ is a ($\mathcal{Z}$-)prekernel of $g$ and $g$ is a ($\mathcal{Z}$-)precokernel of $f$.

\end{definition}

\begin{remark} \label{remark:canonical short exact sequences}

The Proposition \ref{prop:Galois_connection} gives canonical ways of constructing short $\mathcal{Z}$-preexact sequences, in a category with $\mathcal{Z}$-prekernels and $\mathcal{Z}$-precokernels (cf. \cite[\S VIII.1]{SM:cat}, for the classic and similar case of kernel and cokernel).

\end{remark}

\begin{definition} \label{def:pretorsion theory}

The pair $(\mathcal{T},\mathcal{F})$ is a pretorsion theory provided the following two conditions are satisfied:
\begin{enumerate}
\item $Hom_\mathcal{C} (T,F)=Triv_\mathcal{Z} (T,F)$, for every object $T\in \mathcal{T}$ and every object $F\in \mathcal{F}$;

\item for any object $B$ of $\mathcal{C}$, there is a short $\mathcal{Z}$-preexact sequence

$$
\begin{picture}(100,0)

\put(0,0){$A$}\put(50,0){$B$}\put(100,0){$C$,}

\put(13,5){\vector(1,0){35}}\put(63,5){\vector(1,0){35}}

\put(25,12){$f$}\put(75,12){$g$}

\end{picture}
$$ with $A\in\mathcal{T}$ and $C\in\mathcal{F}$.
\end{enumerate}

\end{definition}

\begin{remark} \label{remark:uniqueness of short preexact sequence}

The short $\mathcal{Z}$-preexact sequence, given in condition (2) in Definition \ref{def:pretorsion theory} just above, is uniquely determined up to isomorphism (cf. Proposition 3.1 in \cite{pretorsion in general}).

\end{remark}

\begin{example}

The pair $(Equiv, Ord)$ is a pretorsion theory for the category $Preord$, whose objects are the reflexive and transitive relations, and the morphisms are the monotone maps. $Equiv$ denotes the subcategory of equivalence relations, and $Ord$ the subcategory of partial orders, so that $Equiv\bigcap Ord = Set$ is the category of sets. Check \cite{stable category} for details.

\end{example}

\section{Symmetric and antisymmetric categories}
\label{sec:symmetric, antisymmetric categories}

$Cat$ is the category whose objects are the small categories, and whose objects are the functors.

$CatEquiv$ will denote the full subcategory of $Cat$ determined by the \textit{symmetric categories} $\mathbb{T}$, meaning that for any $T,T'\in\mathbb{T}$, if $Hom_\mathbb{T}(T,T')\neq\emptyset$ then $Hom_\mathbb{T}(T',T)\neq\emptyset$.

$CatOrd$ will denote the full subcategory of $Cat$ determined by the \textit{antisymmetric categories} $\mathbb{F}$: for any $F,F'\in\mathbb{F}$, if $Hom_\mathbb{F}(F,F')\neq\emptyset$ and $Hom_\mathbb{F}(F',F)\neq\emptyset$, then $F=F'$.

Therefore, $CatMon=CatEquiv\bigcap CatOrd$ is the full (and replete) subcategory of $Cat$ whose objects are classes of monoids.

The trivial functors in $Triv_{CatMon}(\mathbb{A},\mathbb{B})$ (see Data at the beginning of section \ref{sec:pretorsion_theory}), are going to be characterized in the following Lemma \ref{lemma:trivial functor}.

\begin{lemma}\label{lemma:trivial functor}

The functor $F:\mathbb{A}\rightarrow \mathbb{B}$ is trivial if and only if, for every $A,A'\in\mathbb{A}$, if $Hom_\mathbb{A}(A,A')\neq\emptyset$ then $F(A)=F(A')$.

\end{lemma}
\begin{proof}
If $F:\mathbb{A}\rightarrow \mathbb{B}$ is trivial, by definition $F$ factors through some $\mathbb{C}\in CatMon$:
$$
\begin{picture}(100,0)

\put(0,0){$\mathbb{A}$}\put(50,0){$\mathbb{C}$}\put(100,0){$\mathbb{B}$.}

\put(13,5){\vector(1,0){35}}\put(63,5){\vector(1,0){35}}

\put(25,12){$G$}\put(75,12){$H$}.

\end{picture}
$$
Then, as $\mathbb{C}$ does not have morphisms between distinct objects, it follows trivially that if there exists a morphism $f:A\rightarrow A'$ then $G(A)=G(A')$ and $F(A)=H(G(A))=H(G(A'))=F(A')$.

Conversely, supposing that
$$\forall_{A,A'\in\mathbb{A}} Hom_\mathbb{A}(A,A')\neq\emptyset\Rightarrow F(A)=F(A'),$$
it is obvious that
$$
\begin{picture}(100,0)
\put(-65,0){$F=H\circ G:$}
\put(0,0){$\mathbb{A}$}\put(50,0){$\mathbb{C}$}\put(100,0){$\mathbb{B}$,}

\put(13,5){\vector(1,0){35}}\put(63,5){\vector(1,0){35}}

\put(25,12){$G$}\put(75,12){$H$}.

\end{picture}
$$
with $obj(\mathbb{C})=obj(\mathbb{B})$ ($\mathbb{C}$ and $\mathbb{B}$ have the same objects) and \begin{center}
$Hom_\mathbb{C}(B,B')=\left\{ \begin{array}{lll} Hom_\mathbb{B}(B,B')\ if\ B=B' \\ \\

                            \emptyset\ \ \ \ otherwise,
                            \end{array}  \right.$
\end{center}
$H$ being the inclusion functor, and $G$ the restriction of the functor $F$ to the codomain $\mathbb{C}$.
Finally, notice that $\mathbb{C}\in CatMon$.
\end{proof}

The following Proposition \ref{prop:functors from symetric to antisymetric are trivial} asserts condition (1) in Definition \ref{def:pretorsion theory}. Notice that, in order to show that $(CatEquiv,CatOrd)$ is a pretorsion theory for $Cat$, it only remains to check condition (2) in Definition \ref{def:pretorsion theory}.

\begin{proposition}\label{prop:functors from symetric to antisymetric are trivial}
$$Hom_{Cat}(\mathbb{T},\mathbb{F})=Triv_{CatMon}(\mathbb{T},\mathbb{F}),$$ for every $\mathbb{T}\in CatEquiv$ and every $\mathbb{F}\in CatOrd$.
\end{proposition}

\begin{proof}

The proof follows immediately from the respective definitions of symmetrical and antisymmetrical categories $\mathbb{T}$ and $\mathbb{F}$, and from Lemma \ref{lemma:trivial functor}.

\end{proof}

\section{CatMon-Prekernels and CatMon-Precokernels}
\label{sec:CatMon-Prekernels, CatMon-Precokernels}

\begin{proposition}\label{prop:CatMon-Prekernels}

Let $F:\mathbb{A}\rightarrow\mathbb{A'}$ be any functor in $Cat$.

Then, the functor $K:\mathbb{X}\rightarrow\mathbb{A}$ is a prekernel of $F$, where

\noindent $obj(\mathbb{X})=obj(\mathbb{A})$ ($\mathbb{X}$ and $\mathbb{A}$ have the same objects),

\begin{center}
$Hom_\mathbb{X}(A,A')=\left\{ \begin{array}{lll} Hom_\mathbb{A}(A,A')\ if\ F(A)=F(A') \\ \\

                            \emptyset\ \ \ \ otherwise\ (F(A)\neq F(A'))
                            \end{array}  \right.$
\end{center}

for every $A,A'\in\mathbb{A}$,

\noindent and $K$ is the inclusion functor.

\end{proposition}

\begin{proof}

First, one has to establish that $\mathbb{X}$ is a category:
\begin{itemize}
\item $1_A\in Hom_\mathbb{A}(A,A)=Hom_\mathbb{X}(A,A)$, hence the identity of each object in $\mathbb{A}$ is also in $\mathbb{X}$;
\item consider in $\mathbb{X}$ the composable morphisms $
\begin{picture}(0,0)

\put(0,0){$A$}\put(50,0){$A'$}\put(100,0){$A''$,}

\put(13,5){\vector(1,0){35}}\put(63,5){\vector(1,0){35}}

\put(25,12){$f$}\put(75,12){$g$}

\end{picture}
$

\noindent then, necessarily, by the definition of $\mathbb{X}$,

$F(A)=F(A')=F(A'')\Rightarrow F(A)=F(A'')$

$\Rightarrow g\circ f\in Hom_\mathbb{A}(A,A'')=Hom_\mathbb{X}(A,A'')$,

\noindent hence the composition of any two morphisms of $\mathbb{X}$ is also in $\mathbb{X}$.
\end{itemize}

Secondly, one has to show that $F\circ K$ is a trivial functor (cf. Lemma \ref{lemma:trivial functor}):
consider in $\mathbb{X}$ a morphism $f:A\rightarrow A'$ with $A\neq A'$; one wants to show that $F\circ K (A)=F\circ K (A')$; being $K$ the inclusion functor, $F\circ K(f)=F(f):F(A)\rightarrow F(A')$, and $F(A)=F(A')$ by the construction of $\mathbb{X}$, that is $F\circ K (A)=F\circ K (A')$.\\

Thirdly and finally, one has to check the universal property given in Definition \ref{def:prekernel and precokernel}:
suppose $\Lambda :\mathbb{Y}\rightarrow\mathbb{A}$ is a functor such that $F\circ \Lambda$ is trivial; since $K$ is the inclusion functor, one has to show that $\Lambda(\mathbb{Y})\subseteq K(\mathbb{X})$; suppose by ``reductio ad absurdum" that there is in $\mathbb{Y}$ a morphism $g:Y\rightarrow Y'$ such that $\Lambda(g):\Lambda(Y)\rightarrow \Lambda(Y')$ is not in $K(\mathbb{X})$; then, by the construction of $\mathbb{X}$, $F(\Lambda (Y))\neq F(\Lambda (Y'))$, which contradicts the assumption that $F\circ\Lambda$ is trivial (cf. Lemma \ref{lemma:trivial functor}).

\end{proof}

In the following Proposition \ref{prop:Catmon-Precokernels}, a construction of a precokernel of any functor is given.

\begin{proposition}\label{prop:Catmon-Precokernels}

Let $F:\mathbb{A}\rightarrow \mathbb{A}'$ be a functor in $Cat$.

Consider the well-known adjunction $(G,U,\eta): Graph\rightarrow Cat$, where G is the left-adjoint of the forgetful functor $U$ from $Cat$ into the category of graphs, and $\eta :1_{Graph}\rightarrow UG$ is the unit of the adjunction (see \cite[II.7]{SM:cat}).

Let $\zeta_F$ be the equivalence relation on the set of objects of $\mathbb{A}'$, $obj(\mathbb{A}')$, generated by $\{(F(A_1),F(A_2))|Hom_\mathbb{A}(A_1,A_2)\neq\emptyset; A_1,A_2\in\mathbb{A}\}$.

Consider the graph morphism $(1_{mor(\mathbb{A}')},\pi_{\zeta_F}): U\mathbb{A}'\rightarrow \mathbb{P}$ from the underlying graph of $\mathbb{A}'$ into the graph $\mathbb{P}=(mor(\mathbb{A}'),obj(\mathbb{A}')/\zeta_F)$, where $1_{mor(\mathbb{A}')}$ is the identity on the arrows and $\pi_{\zeta_F}$ is the canonical projection of the set of nodes $obj(\mathbb{A}')$ into its equivalence classes.

Consider the unit morphism of $\mathbb{P}$, $\eta_\mathbb{P}: \mathbb{P}\rightarrow UG\mathbb{P}$, in the adjunction $G\dashv U$.

Consider finally the canonical functor $\Pi :G\mathbb{P}\rightarrow G\mathbb{P}/\equiv$, where $\equiv$ stands for the least congruence (in the sense of \cite[II.8]{SM:cat}) on $G\mathbb{P}$ which makes the graph morphism
$$U\pi =U\Pi\circ \eta_\mathbb{P} \circ (1_{mor(\mathbb{A}')},\pi_{\zeta_F}):U\mathbb{A}'\rightarrow U(G\mathbb{P}/\equiv)$$
a functor from $\mathbb{A}'$ into $G\mathbb{P}/\equiv$.\\

Then, $\pi :\mathbb{A}'\rightarrow G\mathbb{P}/\equiv$ is a precokernel of $F:\mathbb{A}\rightarrow \mathbb{A}'$.

\end{proposition}
\begin{proof}

Let $F':\mathbb{A'}\rightarrow \mathbb{B}$ be a functor such that the composite $F'\circ F$ is trivial. One has to show that there is one and only one functor $H':G\mathbb{P}/\equiv\rightarrow \mathbb{B}$ such that $H'\circ\pi =F'$.\\

Since $F'\circ F$ is trivial, there is one and only one morphism of graphs $\varphi$ such that $\varphi\circ (1_{mor(\mathbb{A}')},\pi_{\zeta_F})=UF':U\mathbb{A}'\rightarrow U\mathbb{B}$. In fact, define $\varphi (g:[A_1']\rightarrow [A_2'])=F'g:F'(A_1')\rightarrow F'(A_2')$. It is well defined since, for instance, if $A_0'\in [A_1']$ then there exists a sequence of morphisms (``zigzag") $A_0\rightarrow \cdots\leftarrow A_n$ such that $F(A_0)=A_0'$ and $F(A_n)=A_1'$, which implies that $F'(A_0')=F'(A_1')$ because $F'\circ F$ is trivial (cf. Lemma \ref{lemma:trivial functor}).

There is also only one functor $H:G(\mathbb{P})\rightarrow \mathbb{B}$ such that $\varphi =UH\circ\eta_\mathbb{P} :\mathbb{P}\rightarrow U\mathbb{B}$, being $\eta_\mathbb{P}$ the unit morphism of the adjunction $U\vdash G:Graph\rightarrow Cat$ (cf. Theorem 1 in \cite[II.7]{SM:cat}).

There is also a unique functor $H'$ from the quotient category $G\mathbb{P}/\equiv$ into $\mathbb{B}$ such that $H'\circ\Pi =H$; in order to prove so, one has to show that (cf. \cite[II.8]{SM:cat}) $H$ identifies $\eta_\mathbb{P}\circ (1_{mor(\mathbb{A}')},\pi_{\zeta_F})(1_{A'})=<1_{A'}>$ and $1_{[A']}$, for every $A'\in\mathbb{A}'$, and that $H$ identifies $\eta_\mathbb{P}((1_{mor(\mathbb{A}')},\pi_{\zeta_F})(g))\circ \eta_\mathbb{P}((1_{mor(\mathbb{A}')},\pi_{\zeta_F})(f))=\eta_\mathbb{P}(g)\circ \eta_\mathbb{P}(f)=<f,g>$ with $\eta_\mathbb{P}((1_{mor(\mathbb{A}')},\pi_{\zeta_F})(g\circ f))=\eta_\mathbb{P}(g\circ f)=<g\circ f>$, for every pair $(g:A_2'\rightarrow A_3',f:A_1'\rightarrow A_2')$ of composable morphisms in $\mathbb{A}'$; this is obvious since $UH\circ\eta_\mathbb{P}\circ (1_{mor(\mathbb{A}')},\pi_{\zeta_F})=UF'$ and $F'$ is a functor.\\

It was proved just above that there is a functor $H'$ such that $H'\circ \pi =F'$. It remains to check that such functor is the unique which satisfies $H'\circ \pi =F'$.\\

Suppose that $S'$ is a functor such that $S'\circ \pi =F'$, then there is a functor $S$ such that $S=S'\circ \Pi$, and then there is a graph morphism $\sigma$ such that $US\circ \eta_\mathbb{P}=\sigma$, with $\sigma\circ (1_{mor(\mathbb{A}')},\pi_{\zeta_F})=UF'$, which implies that $\sigma =\varphi$ as defined above, and so $H'=S'$ going backwards.

\end{proof}

\section{Short CatMon-preexact sequences}
\label{sec:short preexact sequences}

Let $\mathbb{A}'$ be any category, and let $\mathbb{A}$ be a category with the same objects, $obj(\mathbb{A})=obj(\mathbb{A}')$, and such that, for any objects $A,B\in\mathbb{A}'$, if $Hom_\mathbb{A}(A,B)\neq\emptyset$ then $Hom_{\mathbb{A}'}(A,B)\neq\emptyset$ and $Hom_{\mathbb{A}'}(B,A)\neq\emptyset$ and $Hom_{\mathbb{A}}(A,B)=Hom_{\mathbb{A}'}(A,B)$.\\

Let $F:\mathbb{A}\rightarrow \mathbb{A}'$ be the inclusion functor of $\mathbb{A}$ in $\mathbb{A}'$, and $\pi : \mathbb{A}'\rightarrow G\mathbb{P}/\equiv$ the precokernel of F constructed as in Proposition \ref{prop:Catmon-Precokernels}.\\

It is an immediate consequence of the characterization of $CatMon$-prekernel in Proposition \ref{prop:CatMon-Prekernels} that $F$ is the prekernel of $\pi$. so that we have constructed a short preexact sequence

$$
\begin{picture}(100,0)

\put(0,0){$\mathbb{A}$}\put(50,0){$\mathbb{A}'$}\put(100,0){$\mathbb{G\mathbb{P}/\equiv}$.}

\put(13,5){\vector(1,0){35}}\put(63,5){\vector(1,0){35}}

\put(25,12){$F$}\put(75,12){$\pi$}.

\end{picture}
$$
for each $\mathbb{A}'\in Cat$, with $\mathbb{A}\in CatEquiv$.\\

Suppose that, one requires the category $\mathbb{A}$ just defined to satisfy in addition: for every $A,B\in\mathbb{A}'$, if $Hom_{\mathbb{A}'}(A,B)\neq\emptyset$ and $Hom_{\mathbb{A}'}(B,A)\neq\emptyset$ then $Hom_{\mathbb{A}}(A,B)\neq\emptyset$.
Then, since $obj(G\mathbb{P}/\equiv)=obj(\mathbb{A}')/\zeta_F$ and by the nature of $\eta_\mathbb{P}$ and $\Pi$ (cf. Proposition \ref{prop:Catmon-Precokernels} and \cite[II.7,8]{SM:cat}, it is clear that $G\mathbb{P}/\equiv\in CatOrd$.\\

It was proved that, for every category $\mathbb{A}'\in Cat$, there is a short $CatMon$-preexact sequence
$$
\mathbb{A}\rightarrow \mathbb{A}'\rightarrow \mathbb{A}''
$$
with $\mathbb{A}\in CatEquiv$ and $\mathbb{A''}\in CatOrd$.

Hence, the following Theorem can be stated.

\begin{theorem}
\label{theo:pretorsion theory}

The pair $(CatEquiv,CatOrd)$ is a pretorsion theory for the category of all categories $Cat$.

\end{theorem}

\section*{Acknowledgement}
This work was supported by
 The Center for Research and Development in Mathematics and Applications (CIDMA) through the Portuguese Foundation for Science and Technology

(FCT - Funda\c{c}\~ao para a Ci\^encia e a Tecnologia),

references UIDB/04106/2020 and UIDP/04106/2020.

\end{document}